\nonstopmode
\documentclass[10pt, reqno]{amsart}
\usepackage{graphicx}
\usepackage{latexsym}
\usepackage{fancyhdr}
\usepackage{amsmath, amssymb, amsthm}
\usepackage[all]{xy}
\usepackage{pdflscape}
\usepackage{longtable}
\usepackage{rotating}
\usepackage{verbatim}
\usepackage
{hyperref}
\hypersetup{
    linkcolor = blue,
    citecolor = blue,
    urlcolor  = blue,
    colorlinks = true,
}

\usepackage{etoc}
\etocsettocstyle
  {\medskip}          
  {\bigskip}

\etocsetstyle{section}                                          
  {}
  {\leavevmode\leftskip 0pt\relax}            
  {\etocnumber \etocname          
   \dotfill\etocpage\par\smallskip}
  {}                                          

\etocsetstyle{subsection}
  {}
  {\leavevmode\leftskip 1.5em\relax}
  {\etocnumber\ \etocname\dotfill\etocpage\par\smallskip}
  {}

\etocsettocstyle
  {\section*{\contentsname}\small}           
  {\bigskip}                                 

\usepackage{cleveref}
\usepackage{subfigure}
\usepackage{mathrsfs}
\usepackage{mdwlist}
\usepackage{dsfont}

\usepackage{mathtools}
\mathtoolsset{showonlyrefs=true}

\usepackage{float}
\usepackage{color}
\usepackage{stmaryrd}
\usepackage{parskip}
\usepackage{orcidlink}

\usepackage{tkz-base}
\usepackage{pgfplots}
\pgfplotsset{compat=newest}

\usepackage{tkz-euclide}
\usepackage{etoolbox}

\definecolor{teal}{rgb}{0.0, 0.5, 0.5}

\newcounter{mnotecount}[section]

\newcommand{\rmnote}[1]{}

\allowdisplaybreaks

\DeclareFontFamily{U}{mathb}{\hyphenchar\font45}
\DeclareFontShape{U}{mathb}{m}{n}{
      <5> <6> <7> <8> <9> <10> gen * mathb
      <10.95> mathb10 <12> <14.4> <17.28> <20.74> <24.88> mathb12
      }{}
\DeclareSymbolFont{mathb}{U}{mathb}{m}{n}
\DeclareFontSubstitution{U}{mathb}{m}{n}

\theoremstyle{plain}
\newtheorem*{theorem*}{Theorem}
\newtheorem{theorem}{Theorem}[section]
\newtheorem*{lemma*}{Lemma}
\newtheorem{lemma}[theorem]{Lemma}
\newtheorem*{assumption*}{Assumption}

\newtheorem*{proposition*}{Proposition}
\newtheorem{proposition}[theorem]{Proposition}
\newtheorem*{corollary*}{Corollary}
\newtheorem{corollary}[theorem]{Corollary}
\newtheorem*{claim*}{Claim}

\newtheorem*{conjecture*}{Conjecture}

\newtheorem*{question*}{Question}

\newtheorem*{result*}{Result}

\theoremstyle{definition}
\newtheorem*{definition*}{Definition}
\newtheorem{definition}[theorem]{Definition}
\newtheorem*{example*}{Example}

\newtheorem*{algorithm*}{Algorithm}
\newtheorem*{remark*}{Remark}
\newtheorem*{remarks*}{Remarks}
\newtheorem{remark}[theorem]{Remark}

\newtheorem*{convention*}{Convention}

\Crefname{l}{Lemma}{Lemmas}    
\Crefname{p}{Proposition}{Propositions}
\Crefname{t}{Theorem}{Theorems}
\Crefname{c}{Corollary}{Corollaries}
\Crefname{r}{Remark}{Remarks}
\Crefname{d}{Definition}{Definitions}
\Crefname{e}{Example}{Examples}
\Crefname{q}{Question}{Questions}

\numberwithin{equation}{section}

\def\al{\alpha}
\def\be{\beta}

\def\ep{\epsilon}

\def\ze{\zeta}
\def\et{\eta}

\def\rh{\rho}

\def\si{\sigma}

\def\ta{\tau}

\def\om{\omega}

\def\Ga{\Gamma}

\def\C{\mathbb{C}}

\def\N{\mathbb{N}}
\def\Q{\mathbb{Q}}
\def\R{\mathbb{R}}
\def\Z{\mathbb{Z}}

\def\cC{\mathcal{C}}

\def\cL{\mathcal{L}}

\def\cP{\mathcal{P}}

\def\cV{\mathcal{V}}

\def\p{\partial}

\def\<{\langle}
\def\>{\rangle}
\renewcommand{\o}{\circ}

\def\ol{\overline}

\def\nn{\mathbf{N}}

\def\ee{\mathbf{e}}
\def\zz{\mathbf{z}}

\def\ol{\overline}

\def\a#1{\left\llbracket{#1}\right\rrbracket}

\let\on=\operatorname

\newcommand{\sr}[1]%
{\ifmmode{}^\dagger\else${}^\dagger$\fi\ifvmode
\vbox to 0pt{\vss
 \hbox to 0pt{\hskip\hsize\hskip1em
 \vbox{\hsize3cm\raggedright\pretolerance10000
 \noindent #1\hfill}\hss}\vss}\else
 \vadjust{\vbox to0pt{\vss%
 \hbox to 0pt{\hskip\hsize\hskip1em%
 \vbox{\hsize3cm\raggedright\pretolerance10000%
 \noindent \small #1\hfill}\hss}\vss}}\fi%
}

\providecommand{\mapsfrom}{\kern.2em%
\setbox0=\hbox{$\leftarrow$\kern-.10em\rule[0.26mm]{0.1mm}{1.3mm}}\box0%
\kern.3em}

\title[]{B\'ezout bounds and rational points on quasianalytic graphs}

\author[Armin Rainer]{Armin Rainer  \orcidlink{0000-0003-3825-3313}}

\address{Armin Rainer: Faculty of Mathematics and Geoinformation,
    Institute for Statistics and Mathematical Methods in Economics, E105-04,
TU Wien, Wiedner Hauptstraße 8, 1040 Vienna, Austria}
\email{armin.rainer@tuwien.ac.at}

\begin{document}

\begin{abstract}
    For functions $f : [0,1] \to [0,1]$ in quasianalytic Denjoy--Carleman classes $\cC_M$, 
    we establish explicit bounds for the number of intersection points of the graph $\Ga(f)$ 
    with algebraic curves of degree at most $d$. The bounds are given in terms of the Bang 
    degree determined by the weight sequence $M$, the degree $d$, and the norming constant of 
    $\Ga(f)$ with respect to the space of real bivariate polynomials of degree at most $d$. 
    We give instances, where the bounds are at most polynomial in $d$ and, consequently, 
    the graph $\Ga(f)$ contains few rational points: the number of rational points of height 
    at most $H$ is bounded by a polynomial in $\log H$.
\end{abstract}

\thanks{This research was funded in whole or 
    in part by the Austrian Science Fund (FWF)  DOI 10.55776/PAT1381823.
For open access purposes, the author has applied a CC BY public copyright license to any author-accepted manuscript version arising from this submission.}
\thanks{\orcidlinkc{0000-0003-3825-3313}}

\date{October 15, 2026}

\maketitle



\section{Introduction}

Bounding the maximal number of intersection points between the graph $\Ga(f)$ of a transcendental complex analytic function $f$ on a bounded disk 
and complex algebraic curves of degree at most $d$, has attracted considerable interest.  
The case $f(z)=e^z$ was treated by Tijdeman \cite{Tijdeman:1971aa}, and the general case was studied in \cite{Coman:2003aa}, 
\cite{Coman:2007aa}, and \cite{Brudnyi:2008aa}. 
In \cite{Comte:2018aa}, such bounds were called \emph{B\'ezout bounds for $f$}, and sufficient conditions were given,
in the complex analytic setting, for the existence of
B\'ezout bounds that are polynomial in $d$, together with Diophantine applications.

In this paper, we establish B\'ezout bounds for functions in quasianalytic Denjoy--Carleman classes $\cC_M$, i.e.,
regularity classes of $\cC^\infty$ functions  with the unique continuation property, determined by bounds on the infinite sequence of derivatives in terms 
of a weight sequence $M=(M_j)$. For $M= (j!)$, we recover the real analytic class, 
in general, $\cC_M$ is strictly larger.
Let $f : [0,1] \to [0,1]$ be a $\cC_M$ function and $p \in \R[X,Y]\setminus \{0\}$ a real polynomial in two variables. 
If $p(x,f(x))$ is not identically zero, then the number of zeros of $p(x,f(x))$, $x \in [0,1]$, (counted with multiplicities) 
is finite; we denote this number by $\zz(p,f)$. We will see that also
\begin{equation}
    \zz_d(f)  := \sup_{p \in \cP_2^d \setminus \{0\}} \zz(p,f)
\end{equation}
is a finite integer, provided that $f$ is transcendental, i.e., $p(x,f(x))$ is not identically zero for all nonzero $p \in \R[X,Y]$.
We denote by $\cP_n^d = \R[X_1,\ldots,X_n]_d$ the space of real polynomials in $n$ variables of maximal degree $d$.

Our B\'ezout bounds will be given in terms of the \emph{Bang degree}, a natural complexity parameter for functions in quasianalytic 
Denjoy--Carleman classes determined by the weight sequence $M$ and the size of the function. 
It was introduced by Bang \cite{Bang53} and used e.g.\ in \cite{NazarovSodinVolberg04}, \cite{Rainer:2022aa}, and \cite{Rainer:2026aa}.
We obtain an explicit bound for $\zz(p,f)$ which depends only on the class $\cC_M$ of $f$, the degree $d$ of $p$, and the ratio 
$\|p\|_{[0,1]^2}/\|p\|_{\Ga(f)}$ (\Cref{t:Bezout1}), where $\|p\|_K := \sup_{x\in K} |p(x)|$.
The proof easily generalizes to give a similar bound for the number of zeros of $p(x,f_1(x),\ldots,f_n(x))$, $x \in [0,1]$,  
for $\cC_M$ functions $f_i : [0,1] \to [0,1]$, $1 \le i \le n$, and $p \in \cP_{n+1}^d$, but this is left to the reader. 

The number 
\begin{equation}
    \ee_d(f)  := \sup_{p \in \cP_2^d \setminus \{0\}} \log \Big(\frac{\|p\|_{[0,1]^2}}{\|p\|_{\Ga(f)}}\Big)
\end{equation}
is another interesting transcendence measure of $f$; $\ee_d(f)$ is a finite nonnegative number, provided that $f$ is transcendental. 
In the terminology of \cite{BRUDNYI:2015vu}, 
$\ee_d(f)$ is the logarithm of the $\cP^d_2$-norming constant of $\Ga(f)$. Note that $\ee_d(f)$ encodes the best constant 
of a Remez inequality for $p \in \cP_2^d$ on a subset $\Ga(f)$ of $[0,1]^2$ of zero measure.

We conclude that 
\begin{equation}
    \zz_d(f) \le F_M(2d\log d + \ee_d(f)), \quad d \ge 1,
\end{equation}
where $F_M$ is an explicit function depending only on the weight sequence $M$ (\Cref{c:zded}).
For instance, for the analytic class $M=(j!)$ the function $F_M$ grows linearly, and if $f : [0,1] \to [0,1]$ is 
transcendental and the restriction of an entire function, then we find
\begin{equation}
    \zz_d(f) \le 12\, d\log d + 6\, \ee_d(f) + C, \quad d \ge 1,
\end{equation}
for a constant $C=C(f)$ (\Cref{c:en}). In the framework of entire functions,
similar connections have been established in \cite{Coman:2003aa}, 
\cite{Coman:2007aa}, and \cite{Brudnyi:2008aa}. 

We estimate $F_M$ for a one-parameter family of quasianalytic weights, the so-called \emph{Denjoy weights}, and identify 
precisely the range of parameters for which $F_M$ is bounded by a polynomial (\Cref{p:Denjoy}).
In particular, $F_{M_s}$ is bounded by a polynomial for all $M_s = (j!\, (\prod_{i=1}^j \log(e i))^s)_j$ with $s\in [0,1]$.
Further away from the analytic class, $F_M$ grows like higher and higher towers of exponentials.

Polynomial B\'ezout bounds for $f$ have a striking Diophantine consequence (see \cite{Pila:2006ab}):  
the number of rational points of height at most $H$ in $\Ga(f)$ is bounded by a polynomial in $\log H$. 
Thus, it is interesting to ask when $\ee_d(f)$ is bounded by a polynomial in $d$, 
especially, for transcendental $f : [0,1] \to [0,1]$ of class $\cC_{M_s}$, where $s \in [0,1]$. 
Even in the real analytic class ($s=0$), $\ee_d(f)$ can grow superpolynomially, see e.g.\ \cite[Remark 5.3]{Comte:2018aa} or \Cref{r:sp},
while $\ee_d(f)$ is bounded by a polynomial for exponential polynomials $f$, see \cite[Section 2.7]{BRUDNYI:2015vu}. 

It seems to be difficult to identify non-analytic quasianalytic functions $f$ such that $\ee_d(f)$ is at most polynomial in $d$.
(Note that $\Ga(f)$ is always pluripolar, see \cite{Coman:2005aa}.)
However, we provide a family $\{f_\xi\}$ of functions $f_\xi : [0,1] \to [0,1]$ of class $\cC_M$, parameterized by a 
sequence $\xi = (\xi_n)_{n\ge 1}$ of independent random variables, each uniformly distributed in $[-1,1]$, such that 
$f_\xi$ is almost surely not real analytic and transcendental, and $\ee_d(f)$ is almost surely at most polynomial in $d$ (\Cref{p:as}). 
Interestingly, the proof of this result does not use quasianalyticity, 
but clearly it requires that the class $\cC_M$ contains the real analytic functions as a proper subclass.

As a consequence, specializing to $M=M_s$ for $s \in [0,1]$, 
the number of rational points of height at most $H$ in $\Ga(f_\xi)$ is almost surely bounded by a polynomial in $\log H$ (\Cref{t:rp}). 
In contrast, for a general function $f : [0,1] \to [0,1]$ of class $\cC_{M_s}$,
the number of rational points of height at most $H$ in $\Ga(f)$ is bounded by $C_\ep H^\ep$ for every $\ep>0$, 
by the Pila--Wilkie theorem \cite{Pila:2006aa}, and this bound is sharp for the class of real analytic functions.
Indeed, the expansion of the real field by restricted functions of class $\cC_{M_s}$ is o-minimal, see \cite{RolinSpeisseggerWilkie03}.

The paper is organized as follows. In \Cref{sec:weights}, we recall basics on quasianalytic Denjoy--Carleman classes
and the Bang degree. In \Cref{sec:intersect}, we prove our bound for $\zz(p,f)$. 
In \Cref{sec:norming}, we investigate the relation of this B\'ezout bound for $f$ with $\ee_d(f)$, 
the (logarithm of the) $\cP_2^d$-norming constant of $\Ga(f)$.
Finally, in \Cref{sec:rational}, we find instances of polynomial B\'ezout bounds for quasianalytic functions and 
discuss the consequences for the number of rational points on corresponding function graphs.

\section{Weights, quasianalytic functions, and the Bang degree} \label{sec:weights}

We recall basics on quasianalytic Denjoy--Carleman classes and introduce our terminology.

\subsection{Quasianalytic Denjoy--Carleman classes}

Let $M=(M_j)_{j\ge 0}$ be a sequence of positive reals. The \emph{Denjoy--Carleman class} $\cC_M$ is defined as follows.
For open $U \subseteq \R^n$, $\cC_M(U)$ consists of all $f \in \cC^\infty(U)$ such that for all compact $K \subseteq U$ 
there are constants $C,\rh>0$ such that 
\begin{equation} \label{eq:DC}
    \|\p^\al f\|_K \le C \rh^{|\al|} M_{|\al|}, \quad \al \in \N^n.
\end{equation}
It will be convenient for us to write 
\begin{equation}
    M_j = M_0 \prod_{i = 1}^j \mu_i, \quad j \ge 1, 
\end{equation}
where $\mu_j = \frac{M_j}{M_{j-1}}$. We will require that the $\mu_j$ are increasing.

\begin{definition}[Admissible weights]
    An \emph{admissible weight} is an infinite increasing sequence $\mu^{|\infty}=(\mu_j)_{j \ge 1}$ of elements in $\R_{>0}$, i.e.,
    $0< \mu_1 \le \mu_2 \le \cdots$.
\end{definition}

A pair $(\mu^{|\infty},M_0)$, where $\mu^{|\infty}$ is an admissible weight and $M_0>0$, uniquely determines $M=(M_j)_{j\ge 0}$ 
which satisfies the logarithmic convexity property 
\begin{equation}
    M_j^2 \le M_{j-1}M_{j+1},\quad j \ge 1,
\end{equation}
and vice versa.

\begin{definition}[Quasianalytic weights]
    An admissible weight $\mu^{|\infty} = (\mu_j)_{j \ge 1}$ is called \emph{quasianalytic} if 
    \begin{equation} \label{eq:qa}
    \sum_{j \ge 1}  \frac{1}{\mu_j} = \infty.
\end{equation}
\end{definition}

By the Denjoy--Carleman theorem, the class $\cC_M$ is quasianalytic (i.e., the infinite Taylor series homomorphism 
$\cC_M(U) \to \R\a{X_1,\ldots,X_n}$ is injective at any point of $U$, provided that $U$ is connected) if and only if \eqref{eq:qa} holds.
See \cite{Rainer:2021aa} for a comprehensive survey of Denjoy--Carleman classes.

We will work with the following definition.

\begin{definition}[$(\mu^{|N},M_0)$-smooth functions]
    Let $\mu^{|\infty} = (\mu_j)_{j\ge 1}$ be an admissible weight and $M_0>0$.
    Let $N \in \N_{\ge 1} \cup \{\infty\}$.
    A function $f : I \to \R$, defined on a compact interval $I \subseteq \R$,
    is called \emph{$(\mu^{|N},M_0)$-smooth} if $f$ is of class $\cC^{N}$ 
    and 
    \begin{equation}
        \|f^{(j)}\|_{I} \le M_j, \quad 0 \le j < N +1.
    \end{equation}
\end{definition}

Note that, compared to \eqref{eq:DC}, this is no restriction because the constants $C,\rh$ can be absorbed by redefining $M$.

\subsection{The Bang degree \texorpdfstring{$\nn$}{N}}

We define an integer $\nn$ that plays the role of a complexity parameter for $(\mu^{|N},M_0)$-smooth functions.
It is a variant of an invariant introduced by \cite{Bang53}, later called \emph{Bang degree} by \cite{NazarovSodinVolberg04}.

In the following, $\lceil r \rceil_\N$ denotes the smallest nonnegative integer $j$ satisfying $j\ge r$.

\begin{definition} \label[d]{d:nn}
    Let $\mu^{|\infty}= (\mu_j)_{j\ge 1}$ be an admissible weight and $M_0,a,b>0$.
    We say that $(\mu^{|\infty},M_0,a,b)$ is \emph{admissible}, if 
    \begin{equation} \label{eq:condition}
        \sum_{j=j_0+1}^\infty \frac{1}{\mu_j} > a\, e, \quad \text{ where }  j_0 = \lceil \log (\tfrac{M_0}{b})\rceil_{\N}.
    \end{equation}
    In that case, we define the integer
    \begin{equation} \label{eq:nn}
        \nn= \nn(\mu^{|\infty},M_0,a,b):= 2\, \sup\Big\{n \in \N : \sum_{j=j_0+1}^n \frac{1}{\mu_j} < a\, e\Big\}
    \end{equation}
    and say that $\nn$ is the \emph{degree associated with $(\mu^{|\infty},M_0,a,b)$}.
\end{definition}

If $\mu^{|\infty}$ is quasianalytic, then \eqref{eq:condition} is always satisfied, for all choices of $M_0,a,b>0$. 
But some of our results (e.g.\ \Cref{t:Bezout1}) hold also for non-quasianalytic weights, provided that 
\eqref{eq:condition} is satisfied.

\begin{remark} \label[r]{r:nnprop}
    (1) By definition, for $r,s>0$, $(r\, \mu^{|\infty},s\, M_0,a,b)$ is admissible if and only 
    if $(\mu^{|\infty},M_0,r\, a, \frac{b}{s})$ is admissible,
    and
    \[
        \nn(r \, \mu^{|\infty},s\, M_0,a,b) = \nn(\mu^{|\infty},M_0,r\, a, \tfrac{b}{s}).
    \]

    (2) Assuming admissibility, the integer $\nn(\mu^{|\infty},M_0,a,b)$ is increasing in $a$ and decreasing in $b$.
\end{remark}

\section{On the number of intersection points with algebraic curves} \label{sec:intersect}

The goal of this section is \Cref{t:Bezout1} which provides a bound for the number of intersection points of the graph 
of $\Ga(f)$ of an $(\mu^{|N},M_0)$-smooth function $f : [0,1] \to [0,1]$ with an algebraic curve in terms of the 
Bang degree and certain data of the algebraic curve.

We need an additional assumption on the admissible weight $\mu^{|\infty} = (\mu_j)_{j\ge 1}$: we will assume that  
    the associated sequence 
    \begin{equation} \label{eq:aj}
        a_j := \Big(\frac{\mu_1 \cdots \mu_j}{j!}\Big)^{1/j}
    \end{equation}
    is almost increasing with constant $H\ge 1$, i.e., 
\begin{equation} \label{eq:ai}
    a_i \le H \, a_j\quad \text{ if } i \le j.  
\end{equation}

For convenience,
given $M_0>0$ and hence $M_1 = M_0\mu_1$, we put
    \begin{equation} 
        \ol m := M_0\sqrt{1+\tfrac{1}{M_1^2}}.
    \end{equation}

\begin{theorem} \label[t]{t:Bezout1} 
    Let $\mu^{|\infty} = (\mu_j)_{j\ge 1}$ be an admissible weight satisfying \eqref{eq:ai} with constant $H\ge 1$. 
    Let $M_0>0$, $B_0\ge b_0 > 0$, and $d \in \N_{\ge 1}$.
    Assume that $(\mu^{|\infty},\tfrac{1}{2} e^{4H\ol m}\, d^{2d} B_0,2,b_0)$ is admissible and let 
    \[ 
        \nn = \nn(\mu^{|\infty},\tfrac{1}{2} e^{4H\ol m} d^{2d} B_0,2,b_0).
    \]
Let $f : [0,1] \to [0,1]$ be $(\mu^{|\nn},M_0)$-smooth 
and  
$p \in \cP_2^d$ satisfy $\|p\|_{\Ga(f)} \ge b_0$ and $\|p\|_{[0,1]^2} \le B_0$.
Then 
\begin{equation}
    \zz(p,f)\le \nn,
\end{equation}
where $\zz(p,f)$ denotes the total number of zeros of $p(x,f(x))$, for $x \in [0,1]$, counted with multiplicities. 
\end{theorem}

Observe that in the theorem $f$ is only assumed to be of class $\cC^\nn$ and to satisfy $\|f^{(j)}\|_{[0,1]} \le M_j$ for $j \le \nn$. 
The proof of \Cref{t:Bezout1} is based on the following proposition which is a variant of a result of Bang \cite{Bang53}.

\begin{proposition}[{\cite[Corollary 3.6 and Remark 3.7]{Rainer:2022aa}}] \label[p]{p:zeros}
   Let $I \subseteq \R$ be a compact interval.
   Let $\mu^{|\infty} = (\mu_j)_{j \ge 1}$ be an admissible weight and $M_0>b_0>0$.
   Assume that $(\mu^{|\infty},M_0,|I|,b_0)$ is admissible and let 
    \[ 
        \nn = \nn(\mu^{|\infty},M_0,|I|,b_0).
    \]
   Let $f : I \to \R$ be a $(\mu^{|\nn}, M_0)$-smooth function satisfying $\|f\|_I \ge b_0$.
   Then the total number of zeros of $f$ in $I$ (counted with multiplicities) is finite and 
   bounded by $\nn$.
\end{proposition}

In view of \Cref{p:zeros}, 
it suffices to show that
the composite  function $h(x) = p(x,f(x))$, $x \in [0,1]$, is 
    $(2\, \mu^{|\nn}, \tfrac{1}{2} e^{4 H \ol m} d^{2d} B_0)$-smooth. 
To this end, we will prove \Cref{p:composition}.

\begin{proposition} \label[p]{p:composition}
    Let $\mu^{|\infty} = (\mu_j)_{j\ge 1}$ be an admissible weight satisfying \eqref{eq:ai} with constant $H\ge 1$. 
    Let $M_0>0$.
    If $f : [0,1] \to [0,1]$ is $(\mu^{|N},M_0)$-smooth and $p \in \cP_2^d$, 
    then $h := p \o \ol f : [0,1] \to \R$, where $\ol f : [0,1] \to [0,1]^2$, $\ol f(x) = (x,f(x))$, 
    is $(2\, \mu^{|N}, \tfrac{1}{2} e^{4H\ol m}\,  d^{2d}\, \|p\|_{[0,1]^2})$-smooth. 
\end{proposition}

We begin with a few preparatory results.

\begin{lemma} \label[l]{l:rai}
    Let $\mu^{|\infty} = (\mu_j)_{j\ge 1}$ be an admissible weight satisfying \eqref{eq:ai} with constant $H\ge 1$. 
     Then the sequence $b_j :=a_j^j$ satisfies  
    \begin{equation}
        b_{j_1}b_{j_2} \cdots b_{j_\ell} \le H^\ell \, b_k 
    \end{equation}
    whenever $j_1 + \cdots + j_\ell = k$ and $j_i \ge 1$ for all $1 \le i \le \ell$. 
\end{lemma}

\begin{proof}
    By \eqref{eq:ai}, we have $b_i \le H \, b_{j}^{i/j}$ if $i\le j$. 
    Thus, if $j_1 + \cdots + j_\ell = k$ and $j_i \ge 1$ for all $1 \le i \le \ell$, then 
    \begin{equation}
        b_{j_1}b_{j_2} \cdots b_{j_\ell} \le H\, b_k^{j_1/k}\, H\, b_k^{j_2/k} \cdots H\, b_k^{j_\ell/k} 
        = H^\ell \, b_k 
    \end{equation}
    as claimed.
\end{proof}

\begin{lemma} \label[l]{l:Markov}
    For all $p \in \cP_n^d$ and $j \in \N$,
    \begin{equation}
        \sup_{x \in [0,1]^n}\|d^j p(x)\|_{L^j(\R^n,\R)} \le (2nd^2)^j \|p\|_{[0,1]^n}.
    \end{equation}
\end{lemma}

\begin{proof}
Markov's inequality implies 
\begin{equation}
    \|\p^\al p\|_{[0,1]^n} \le (2d^2)^{|\al|}\, \|p\|_{[0,1]^n},
\end{equation}
for all $\al \in \N^n$.
Then, for $v^i= (v^i_1,\ldots,v^i_n) \in \R^n$ with $|v^i|\le 1$ and $x \in [0,1]^n$, 
\begin{align*}
    |(d^jp)(x)(v^1,\ldots,v^j)| &= \Big|\sum_{i_1=1, \ldots,i_j=1}^n  \p_{i_1}\cdots\p_{i_j} p(x) \, v^1_{i_1}\cdots v^j_{i_j}\Big| 
                                \le (2nd^2)^{j}\, \|p\|_{[0,1]^n}.
\end{align*}
The lemma is proved.
\end{proof}

\begin{proposition}[Fa\`a di Bruno's formula; cf.\ {\cite[Theorem 1]{Yamanaka89}}] \label[p]{p:FdBop}
    Let $f : U \to V$ and $g : V \to W$ be $\cC^k$ maps, where $U,V,W$ are open subsets of normed spaces.
    Then, for $x \in U$,
    \begin{equation}
        \frac{d^k(g \o f)(x)}{k!} = \on{sym} \,\sum  \frac{d^jg(f(x))}{j_1!\cdots j_k!}  \Big( \frac{d^1 f(x)}{1!}\Big)^{j_1}   \cdots \Big( \frac{d^k f(x)}{k!}\Big)^{j_k},
    \end{equation}
    where $j = j_1+ \cdots+j_k$, the sum is taken over all $j_1,\ldots,j_k$ such that 
    $j_1 + 2j_2 + \cdots + kj_k = k$, and
    $\on{sym}$ denotes symmetrization of a multilinear operator.
\end{proposition}

Let us now prove \Cref{p:composition}.
    Clearly, $\|h\|_{[0,1]} \le \|p\|_{[0,1]^2}$.
    By \Cref{p:FdBop}, for $1 \le k \le N$ and $x \in [0,1]$,  
    \begin{equation}
        \frac{|h^{(k)}(x)|}{k!} \le  \,\sum  \frac{\|d^jp(x,f(x))\|_{L^j(\R^2,\R)}}{j_1!\cdots j_k!}  
        \Big( \frac{|\p^1 \ol f(x)|}{1!}\Big)^{j_1}   \cdots \Big( \frac{|\p^k \ol f(x)|}{k!}\Big)^{j_k},
    \end{equation}
    where $j = j_1+ \cdots+j_k$ and the sum is taken over all $j_1,\ldots,j_k$ such that 
    $j_1 + 2j_2 + \cdots + kj_k = k$. 
    Since $\deg p \le d$, we have $d^j p =0$ for $j\ge d+1$. Thus 
    it is enough to take the sum on the right-hand side over all $j_1,\ldots,j_k$ satisfying
    $j_1 + 2j_2 + \cdots + kj_k = k$ and $j_1+\cdots + j_k =j \le d$; 
    let us write $\sum_{j \le d}$ for this restricted sum.

    Since $f$ is $(\mu^{|N},M_0)$-smooth and
    \begin{align*}
        \p^j\ol f(x) = 
        \begin{cases}
            (1,f'(x)) & \text{ if } j=1,
            \\
            (0,f^{(j)}(x)) & \text{ if } j\ge 2,
        \end{cases}
    \end{align*}
    we have, for $x \in [0,1]$,
    \[
        |\p^1\ol f(x)| = \sqrt{1+|f'(x)|^2} \le \sqrt{1+M_1^2} =  M_1 \sqrt{1+\tfrac{1}{M_1^2}} 
    \]
    and, if $j \ge 2$,
    \[
        \frac{|\p^j \ol f(x)|}{j!} = \frac{|f^{(j)}(x)|}{j!} \le \frac{M_j}{j!}.
    \]
    Writing $m_j := M_j/j!$ we find,
    by \Cref{l:rai}, using its notation,
    \begin{align*}
        m_1^{j_1}   \cdots m_k^{j_k} = M_0^j  b_1^{j_1}   \cdots b_k^{j_k} \le M_0^j H^j b_k,
    \end{align*}
    since $j_1+\cdots + j_k =j$ and $j_1 + 2j_2 + \cdots + kj_k = k$.
    By \Cref{l:Markov}, we conclude
    \begin{align} \label{eq:est}
        \frac{|h^{(k)}(x)|}{k!} 
    &\le \|p\|_{[0,1]^2} \, \sum_{j\le d}  \frac{(4d^2)^j}{j_1!\cdots j_k!}  \Big(\sqrt{1+\tfrac{1}{M_1^2}}\Big)^{j_1} m_1^{j_1}   \cdots m_k^{j_k}
    \\
    &\le \|p\|_{[0,1]^2} \, b_k\,  \sum_{j \le d}  \frac{\Big(4\, d^2 H M_0\sqrt{1+\tfrac{1}{M_1^2}}\Big)^j}{j_1!\cdots j_k!}
    \\
    &\le  e^{4 H \ol m}\, d^{2d}\, \|p\|_{[0,1]^2} \, b_k\,  \sum_{j \le d}  \frac{j!}{j_1!\cdots j_k!}  \qquad (\text{using $(4 H \ol m)^j \le e^{4 H \ol m} j!$})
       \\
       &\le e^{4 H \ol m}\, d^{2d}\, \|p\|_{[0,1]^2} \,2^{k-1}\, b_k, 
    \end{align}
    using $\sum  \frac{j! }{j_1!\cdots j_k!} = 2^{k-1}$ in
    the last inequality; see e.g.\ \cite[Lemma 1.4.1]{KrantzParks02}.
    This implies 
    \begin{equation}
        \|h^{(k)}\|_{[0,1]} \le \tfrac{1}{2} e^{4 H \ol m}\, d^{2d}\, \|p\|_{[0,1]^2} \, 2^{k} \mu_1 \cdots \mu_k.
    \end{equation}
    The proof of \Cref{p:composition} and thus of \Cref{t:Bezout1} is complete.

    The following is a variant of \Cref{p:composition} which gives a slightly better bound in \Cref{t:Bezout1} for small $d$.

\begin{proposition} \label[p]{p:composition2}
    In the setting of \Cref{p:composition},
    the composite
    $h = p \o \ol f : [0,1] \to \R$
    is $(2\, \mu^{|N}, \tfrac{1}{2}  (4 d^2 H \hat m)^d\, \|p\|_{[0,1]^2})$-smooth, where
    \begin{equation} \label{eq:olm2}
        \hat m := \max \Big\{  M_0\sqrt{1+\tfrac{1}{M_1^2}},1\Big\}.
    \end{equation}
\end{proposition}

\begin{proof}
    It suffices to modify the estimate \eqref{eq:est} in the following way:
    \begin{align*}
        \frac{|h^{(k)}(x)|}{k!} 
    &\le \|p\|_{[0,1]^2} \, b_k\,  \sum_{j \le d}  \frac{(4\, d^2 H \hat m)^j}{j_1!\cdots j_k!}
\le (4\, d^2 H \hat m)^d\, \|p\|_{[0,1]^2} \, b_k\,  \sum  \frac{j! }{j_1!\cdots j_k!}.
   \end{align*}
   Then the proof of \Cref{p:composition} implies the statement.
\end{proof}

\begin{remark}
   An example, where $d$ is small, is the number of fixed points of $f$.  
    Indeed, the polynomial $p(X,Y) = X-Y$ satisfies $\|p\|_{[0,1]^2} = 1$ and $\deg p =1$.
    Assume that $(\mu^{|\infty},M_0(1+\frac{1}{M_1}),1,b_0)$ is admissible and let 
    \[ 
        \nn = \nn(\mu^{|\infty}, M_0(1+\tfrac{1}{M_1}) ,1,b_0).
    \]
    Let $f : [0,1] \to [0,1]$ be $(\mu^{|\nn},M_0)$-smooth with $\sup_{x\in [0,1]} |x-f(x)|\ge b_0$. 
    By a direct computation,
    the composite function $h(x) = p(x,f(x)) = x-f(x)$, $x \in [0,1]$, is 
    $(\mu^{|\nn}, M_0(1+\tfrac{1}{M_1}))$-smooth. (Alternatively, one can use \Cref{p:composition2} which 
    however gives weaker bounds in this case.)
    Moreover, $\|h\|_{[0,1]}\ge b_0$. 
    By \Cref{p:zeros}, the total number of fixed points of $f$ (counted with multiplicities) 
is finite and bounded by $\nn$.
\end{remark}

\section{Relation to the norming constant} \label{sec:norming}

\subsection{Norming sets and the norming constant}

We recall a definition of \cite{BRUDNYI:2015vu}.

\begin{definition}[{\cite[Definition 1.3]{BRUDNYI:2015vu}}]
    Let $\cV \subseteq C^0([0,1]^n,\R)$ be a finite dimensional subspace.
    A compact subset $Z \subseteq [0,1]^n$ is said to be \emph{$\cV$-norming} if there is a constant $C>0$ 
    such that 
    \begin{equation}
        \|v\|_{[0,1]^n} \le C \, \|v\|_Z, \quad v \in \cV.
    \end{equation}
    The infimum of all such constants $C$ is denoted by $N_{\cV}(Z)$ and called the \emph{$\cV$-norming constant of $Z$}.
This is an inequality of Remez type on a set of zero measure.
\end{definition}

Let $f : [0,1]\to [0,1]$ be a transcendental function, i.e.,  $p(x,f(x)) \not\equiv 0$ for all nonzero $p \in \R[X,Y]$. 
Then the graph $\Ga(f)$ of $f$ is $\cP_2^d$-norming for every $d$; see \cite[Proposition 1.4]{BRUDNYI:2015vu}. 
Let us define 
\begin{align*}
    \ee(p,f) &:= \log\Big(\frac{\|p\|_{[0,1]^2}}{\|p\|_{\Ga(f)}}\Big),
\end{align*}
for $p \in \cP_2^d \setminus \{0\}$,
and
\begin{equation}
    \ee_d(f) :=\sup_{p \in \cP_2^d \setminus \{0\}} \ee(p,f) = \log N_{\cP_2^d}(\Ga(f)).
\end{equation}
Then $\ee(p,f)$ and $\ee_d(f)$ are nonnegative numbers. 

\subsection{B\'ezout bounds in terms of the norming constant}

Let $f : [0,1] \to [0,1]$ be transcendental and $p \in \cP_2^d \setminus \{0\}$.
Recall that $\zz(p,f)$ is the total number of zeros of $p(x,f(x))$, for $x \in [0,1]$, (counted with multiplicities).
Define
\begin{equation}
    \zz_d(f) := \sup_{p \in \cP_2^d \setminus \{0\}} \zz(p,f).
\end{equation}

Let $\mu^{|\infty} = (\mu_j)_{j\ge 1}$ be an admissible quasianalytic weight and $M_0>0$.
Let $f : [0,1] \to [0,1]$ be a transcendental $(\mu^{|\infty},M_0)$-smooth function. 
We will see that $\zz_d(f)$ is finite and 
can be bounded by an explicit function of $d$ and the logarithm of the norming constant $\ee_d(f)$. 
This explicit function only depends on $(\mu^{|\infty},M_0)$.

To this end, we define $\et_{\mu^{|\infty}} : [0,\infty) \to \N$ by  
\begin{equation}
    \et_{\mu^{|\infty}} (t) := \nn(\mu^{|\infty},e^t, 2,1).
\end{equation}
By the definition of $\nn$, we have $\et_{\mu^{|\infty}} (t) = \et_{\mu^{|\infty}} (\lceil t \rceil)$. 
Note that the function $\et_{\mu^{|\infty}}$ is increasing (see \Cref{r:nnprop}).

\begin{theorem} \label[t]{c:zded}
    Let $\mu^{|\infty} = (\mu_j)_{j\ge 1}$ be an admissible quasianalytic weight satisfying \eqref{eq:ai} with constant $H\ge1$. 
    Let $M_0>0$.
If $f : [0,1] \to [0,1]$ is a transcendental $(\mu^{|\infty},M_0)$-smooth function, 
then, for each $p \in \cP_2^d\setminus \{0\}$, 
\begin{equation} \label{eq:zep}
    \zz(p,f) \le \et_{\mu^{|\infty}}\big( 2d \log d + \ee(p,f) + 4H \ol m -\log 2 \big).
\end{equation}
Consequently,
\begin{equation} \label{eq:ze}
    \zz_d(f) \le  \et_{\mu^{|\infty}} \big( 2d \log d + \ee_d(f) + 4H \ol m -\log 2 \big), \quad d \ge 1.
\end{equation}
\end{theorem}

\begin{proof}
    By (the proof of) \Cref{t:Bezout1},
    \begin{equation} 
        \zz(p,f) \le \nn(\mu^{|\infty},\tfrac{1}{2} e^{4 H \ol m} d^{2d}\,\|p\|_{[0,1]^2},2,\|p\|_{\Ga(f)}).
    \end{equation}
    Now, by the properties of $\nn$ (see \Cref{r:nnprop}),
    \begin{align}
        \nn(\mu^{|\infty},\tfrac{1}{2} e^{4 H \ol m} d^{2d}\,\|p\|_{[0,1]^2},2,\|p\|_{\Ga(f)}) 
        &=  
        \nn\Big(\mu^{|\infty},\tfrac{1}{2} e^{4 H \ol m} d^{2d}\,\frac{\|p\|_{[0,1]^2}}{\|p\|_{\Ga(f)}},2,1\Big)
        \\
        &=
        \et_{\mu^{|\infty}}\big(2d \log d + \ee(p,f) + 4H \ol m -\log 2\big)
    \end{align}
    implying \eqref{eq:zep}.
    Since $\et_{\mu^{|\infty}}$ is increasing, \eqref{eq:ze} is immediate from \eqref{eq:zep}.
\end{proof}

Relations similar to \eqref{eq:zep} and \eqref{eq:ze} have been established in \cite{Coman:2003aa}, 
\cite{Coman:2007aa}, and \cite{Brudnyi:2008aa} in the framework of entire functions. 
In the following, we analyze $\et_{\mu^{|\infty}}$ for particular weights.

\subsection{Real analytic and entire transcendental functions}

Let $f : [0,1] \to [0,1]$ be real analytic.
Then there exist constants $C,\rh>0$ such that
$\|f^{(j)}\|_{[0,1]} \le C \rh^j j!$ for all $j \ge 0$. 
So $f$ is $(\rh\, \mu_{\on{an}}^{|\infty},C)$-smooth, where $\mu_{\on{an}}^{|\infty} := (j)_{j \ge 1}$.
By \Cref{p:Denjoy},  
\begin{equation} \label{eq:an}
    e^{2e\rh-1}(m+1)-1 \le \tfrac{1}{2}\et_{\rh\, \mu_{\on{an}}^{|\infty}}(m) \le e^{2e\rh+1} (m+1) - 1, \quad m \in \N_{\ge 1}.
\end{equation}

Thus, \Cref{c:zded} implies the following corollary.

\begin{corollary}
Let $f : [0,1] \to [0,1]$ be real analytic and transcendental. 
There exist $C,\rh>0$ (depending on $f$) such that 
\begin{equation} \label{eq:an1}
    \zz_d(f) \le 2\, e^{2e\rh+1} \Big(2d\log d + \ee_d(f)  + 4 C\sqrt{1 + \tfrac{1}{C^2\rh^2}} +1 -\log 2\Big), \quad d\ge 1.
\end{equation}
\end{corollary}

\begin{proof}
    In this case, $H=1$, $M_0 = C$, and $M_1 = C\rh$. 
    The bound follows from \eqref{eq:ze} and \eqref{eq:an}.
\end{proof}

If $f : [0,1] \to [0,1]$ is the restriction of an entire function, then 
for each $\ta>0$ there exists $C_\ta>0$ such that $\|f^{(j)}\|_{[0,1]} \le C_\ta \ta^j j!$ for all $j \ge 0$.

\begin{corollary} \label[c]{c:en}
Let $f : [0,1] \to [0,1]$ be transcendental and the restriction of an entire function. 
There exists $C=C(f)>0$ such that 
\begin{equation}
    \zz_d(f) \le 12\,d\log d + 6\, \ee_d(f)  + C, \quad d\ge 1.
\end{equation}
\end{corollary}

\begin{proof}
    Choose $\ta>0$ such that $2e\ta +1 \le \log 3$ and apply \eqref{eq:an1} for $\rh = \ta$. 
\end{proof}

We can combine this with a result which is a consequence of \cite[Theorem 2.5]{Brudnyi:2008aa}; see also \cite[Theorem 1.1]{Coman:2007aa} for the special case of $n=1$ and $f$ being of finite order.

\begin{theorem}[{\cite[Theorem 2.9]{BRUDNYI:2015vu}}] \label[t]{t:BY}
    Let $f$ be a non-polynomial entire function on $\C^n$ that is real on $\R^n$.
    Assume that the graph $\Ga(f|_{\R^n})$ of $f|_{\R^n}$ 
    intersects $[0,1]^{n+1}$ 
    in a subset $Z$ of real dimension $n$.
    Then there exists a sequence of positive integers $d_j \to \infty$ and a sequence of positive reals $\ep_j \to 0$ 
    such that 
    \begin{equation}
    \log N_{\cP_{n+1}^{d_j}}(Z) \le d_j^{2+\ep_j},\quad  \text{ for all } j.
    \end{equation}
    The sequence $(d_j)$ depends only on $f$ and the sequence $(\ep_j)$ only on $f$, $n$, and the $n$-dimensional Hausdorff measure of $Z$.
\end{theorem}

\begin{corollary}
Let $f : [0,1] \to [0,1]$ be transcendental and the restriction of an entire function that is real on $\R$. 
There exists a sequence of positive integers $d_j \to \infty$, a sequence of positive reals $\ep_j \to 0$, 
and a constant
$C>0$ such that 
\begin{equation}
    \zz_{d_j}(f) \le 6\, d_j^{2+ \ep_j} + 12\, d_j \log d_j + C, \quad  \text{ for all } j.
\end{equation}
\end{corollary}

\begin{proof}
    Combine \Cref{c:en} with \Cref{t:BY}.
\end{proof}

In general, one cannot expect that $\zz_d(f)$ is bounded by a polynomial in $d$, for all degrees $d\ge 1$.

\begin{remark} \label[r]{r:sp}
For each $\ze \in (0,1)$,
    there exists an analytic function $f$ defined in a neighborhood of $[0,1]$ 
    and a sequence of positive integers $d_j \to \infty$ satisfying 
    \[
        \zz_{d_j}(f) \ge e^{d_j^\ze}, \quad \text{ for all } j.
    \]
    See \cite[Remark 5.3]{Comte:2018aa}. 
    In view of \eqref{eq:an1}, we conclude that 
    \begin{equation} 
        \ee_{d_j}(f) \ge c_1\, e^{d_j^\ze} - 2 d_j \log d_j  - C_2, \quad \text{ for all } j,
    \end{equation}
    for constants $c_1 = c_1(f) \in (0,1)$ and $C_2 = C_2(f) >0$. 
\end{remark}

\subsection{Denjoy weights}

We define the towers of exponentials and logarithms, 
\begin{equation}
    E_0(x) = L_0(x) := x, \quad E_k(x) := \exp (E_{k-1}(x)), \quad 
     L_k(x) := \log (L_{k-1}(x)), 
\end{equation}
for $k \ge 1$,
and set $e_k:= E_k(1)$.

\begin{definition}[Denjoy weights] \label[d]{d:Dw}
Let $s\in [0,\infty)$ and set $k = \lceil s\rceil \in \N$ and $\si = k-s \in [0,1)$.  
We define 
\begin{equation}
    (\mu_s)_j := j L_1(e_kj) L_2(e_k j) \cdots L_{k-1}(e_k j) (L_k(e_k j))^{1-\si}, \quad  j\ge 1.
\end{equation}
Then $\mu_s^{|\infty} := ((\mu_s)_j)_{j \ge 1}$ is an admissible quasianalytic weight for all $s \in [0,\infty)$. 
(Quasianalyticity follows from the integral test, in view of \eqref{eq:it}.)
Since $(\mu_s)_j/j$ is increasing, so is the sequence $a_j$ defined in \eqref{eq:aj}, i.e., \eqref{eq:ai} holds with $H=1$. 

The sequences $\mu_s^{|\infty}$ are called \emph{Denjoy weights}.
Note that $\mu_0^{|\infty}$ coincides with the analytic weight $\mu_{\on{an}}^{|\infty}$.
\end{definition}

Let us investigate the associated function $\et_{\mu^{|\infty}_s}$.
The following proposition extends \cite[Proposition A.1]{Rainer:2026aa}. 

\begin{proposition} \label[p]{p:Denjoy}
    Let $s \ge 0$ and $a\ge 1$. Then $\et_{a\,\mu_s^{|\infty}} : \N_{\ge 1} \to \N$ is bounded by a polynomial  if and only if $s \in [0,1]$.  
    
    For $s=k \in \N$ and $m \in \N_{\ge 1}$, 
    \begin{equation} \label{eq:k} 
        \frac{E_{k+1}(f_k(a,m)-1)}{e_k}-1   \le  \tfrac{1}{2}\et_{a\, \mu_{k}^{|\infty}}(m)  \le \frac{E_{k+1}(f_k(a,m)+1)}{e_k}-1,
\end{equation}
where $f_k(a,m) := 2a e+ L_{k+1}(e_k(m+1))$.
In particular,
\begin{align} \label{eq:k0}
    e^{2ae-1}(m+1)-1 \le &\tfrac{1}{2}\et_{a\, \mu_{0}^{|\infty}}(m) \le e^{2ae+1} (m+1) - 1,
    \\ \label{eq:k1}
  e^{e^{2ae  -1}-1} (m+1)^{e^{2ae -1}}-1  \le &\tfrac{1}{2}\et_{a\, \mu_{1}^{|\infty}}(m) \le e^{e^{2ae  +1}-1} (m+1)^{e^{2ae  +1}}-1.
\end{align}
\end{proposition}

\begin{proof}
    Let us first prove \eqref{eq:k}.
Set
$A^k_n := \sum_{j=1}^n \frac{1}{(\mu_k)_j}$, for $n \ge 1$, and $A^k_0 := 0$.
    Consider $g(t) = L_{k+1}(e_k t)$, for $t \ge 1$. Then
    \begin{align} \label{eq:it}
        g'(t) &= \frac{1}{t L_1(e_kt) L_2(e_k t) \cdots L_k(e_k t)}
    \end{align}
    is decreasing, i.e., $g$ is concave. Thus
    \[
        g(j+1) - g(j) \le g'(j) \le g(j) - g(j-1), 
    \]
    where the left inequality holds for $j\ge 1$ and the right for $j \ge 2$.
    Summing over $j$, we find 
    \[
        g(n+1) - g(1) \le A^k_n \quad \text{ and } \quad A^k_n -A^k_1 \le g(n)-g(1)   
    \]
    which implies 
    \begin{equation} \label{eq:A1}
        L_{k+1}(e_k(n+1)) \le  A^k_n 
    \le L_{k+1}(e_k n) + 1, \quad n \ge 1.
    \end{equation}

    For $x > 1$ define the positive integer $n(x)$ by 
    \begin{equation} \label{eq:A1x}
        A^k_{n(x)} < x \le A^k_{n(x)+1}.
    \end{equation}
    By \eqref{eq:A1} and \eqref{eq:A1x},
    \begin{equation} \label{eq:A1n}
        \frac{E_{k+1}(x-1)}{e_k}- 1  \le  n(x) < \frac{E_{k+1}(x)}{e_k}- 1.
    \end{equation}

    By \eqref{eq:A1x}, we have 
    \begin{equation}
        \sum_{j=m+1}^{n(x)} \frac{1}{(\mu_k)_j}  < x- A^k_{m} \le  \sum_{j=m+1}^{n(x)+1} \frac{1}{(\mu_k)_j}
    \end{equation}
    so that, setting $x = 2ae +A^k_{m}$, we get 
    \[
         \sum_{j=m+1}^{n} \frac{1}{(\mu_k)_j} < 2ae \le \sum_{j=m+1}^{n+1} \frac{1}{(\mu_k)_j} 
    \]
    with $n=n(2ae + A^k_{m})$, since $2ae+ A^k_{m}> 1$.
    Thus, by \eqref{eq:A1} and \eqref{eq:A1n},
    \begin{equation}
        \frac{E_{k+1}(2ae-1 +L_{k+1}(e_k (m+1)))}{e_k}-1  \le n \le \frac{E_{k+1}(2ae+1 +L_{k+1}(e_k (m+1)))}{e_k}-1.
    \end{equation}
    This implies \eqref{eq:k}, since $n=\frac{1}{2} \et_{a\,\mu_k^{|\infty}}(m)$. 

    Now \eqref{eq:k0} and \eqref{eq:k1} easily follow from \eqref{eq:k}.
    Since $(\mu_{s})_j \le (\mu_{t})_j$ for all $j$ if $s \le t$ and thus $\et_{a\, \mu_{s}^{|\infty}} \le \et_{a\, \mu_{t}^{|\infty}}$, 
    we conclude that  $\et_{a\, \mu_{s}^{|\infty}}$ 
    is bounded by a polynomial, for all $s\in [0,1]$.

To see that $\et_{a\, \mu_{s}^{|\infty}}$  cannot be bounded by a polynomial if $s>1$ it is enough to show this for $s \in (1,2)$
and $a=1$.
Fix such $s$, let $\si = 2 -s \in (0,1)$, and
    consider $g(t) = (\log\log(e^e t))^{\si}$, for $t\ge 1$. Since 
    \[
        g'(t) = \frac{\si}{t \log(e^e t) (\log\log(e^et))^{1-\si}}
    \]
    and $g$ is concave, we find similarly to \eqref{eq:A1},
    \begin{equation} \label{eq:ll1}
        (\log \log(e^e (n+1)))^{\si} -1  \le  \si A^s_n \le (\log \log(e^e n))^{\si} - 1 +\tfrac{\si}{e}, \quad n\ge 1, 
    \end{equation}
    where $A^s_n = \sum_{j=1}^n \frac{1}{(\mu_s)_j}$ and noting that $A^s_1 = \frac{1}{e}$. As above,
    let $n(x)$, for $x>1$, be the positive integer satisfying $A^s_{n(x)} < x \le A^s_{n(x)+1}$. Then, by \eqref{eq:ll1},
    \begin{equation}
        e_2(n(x)+1) \ge  E_2((\si x +1 -\tfrac{\si}{e})^{\frac{1}{\si}}) 
    \end{equation}
    and, setting $n = n(2e + A^s_m) =\frac{1}{2} \et_{\mu_{s}^{|\infty}}(m)$, we find
   \begin{align*}
       e_2(n+1) &\ge  E_2((2\si e + \si A^s_m +1 -\tfrac{\si}{e})^{\frac{1}{\si}})      
       \ge E_2(((2 e -\tfrac{1}{e})\si + (L_2(e_2 (m+1)))^{\si} )^{\frac{1}{\si}}).
   \end{align*}
   Suppose for contradiction that $n$ can be bounded by a polynomial in $m$. Then there exist $C\ge 1$ and $N \in \N_{\ge 1}$
   such that, for all $m$, 
   \[
       E_2(((2 e -\tfrac{1}{e})\si  + (L_2(e_2 (m+1)))^{\si} )^{\frac{1}{\si}}) \le C (e_2(m+1))^N,
   \]
   or equivalently,
   \begin{equation} \label{eq:con}
       (2 e -\tfrac{1}{e})\si    \le  (L_2(C (e_2(m+1))^N))^\si - (L_2(e_2 (m+1)))^{\si}.
   \end{equation}
   Setting $z=e_2(m+1)$, we have
   \begin{align}
       L_2(C (e_2(m+1))^N)& = L_2(C z^N) = \log(\log C + N \log z) 
       \\
                                       &= \log\Big(N\log z \Big(1+\frac{\log C}{N \log z}\Big) \Big)
       \\
                                       &= \log\log z + \log N + o(1)
   \end{align}
   as $z \to \infty$.
   By the mean value theorem, the right-hand side of \eqref{eq:con} equals
   \begin{align} \label{eq:mvt}
       (\log\log z + \log N + o(1))^\si - (\log\log z)^{\si} = \si \ze^{\si-1} (\log N + o(1)),
   \end{align}
   for some $\ze$ between $\log\log z + \log N + o(1)$ and $\log\log z$. 
   Since $\si \in (0,1)$, \eqref{eq:mvt} tends to zero as $z \to \infty$ which contradicts \eqref{eq:con}.
\end{proof}

\begin{corollary} \label[c]{c:Denjoy1}
    Let $s \in [0,1]$. 
    If $f : [0,1] \to [0,1]$ is a transcendental $(\mu_s^{|\infty},M_0)$-smooth function, then
    $\zz_d(f)$ is bounded by a polynomial in $d$ and $\ee_d(f)$.
\end{corollary}

\begin{proof}
    This is a consequence of \Cref{c:zded} and \Cref{p:Denjoy}.
\end{proof}

\section{On the number of rational points on quasianalytic graphs} \label{sec:rational}

It is well-known that polynomial B\'ezout bounds give good bounds on the number of rational points (see \cite{Pila:2006ab}). 
In view of \Cref{c:Denjoy1}, it is thus desirable to detect instances where $\ee_d(f)$ is at most polynomial in $d$.

\subsection{Polynomial bounds for \texorpdfstring{$\ee_d(f)$}{ed(f)} and rational points on \texorpdfstring{$\Ga(f)$}{Gamma(f)}}

In \Cref{c:rat}, we adapt the arguments of \cite{Pila:2006ab} to our setting.
Recall that the \emph{height} of a rational number $x \in \Q$ is $\max\{ |p|,|q|\}$, where $p,q \in \Z$ are coprime 
such that $x = p/q$. The height of $(x,y) \in \Q^2$ is the maximum of the heights of $x$ and $y$.

\begin{corollary} \label[c]{c:rat}
    Let $s \in [0,1]$ and let $f : [0,1] \to [0,1]$ be a transcendental function of class 
    $\cC_{M_s}$ associated with the Denjoy weight $\mu_s^{|\infty}$ (with $(M_s)_0=1$)
    such that $\ee_d(f)$ is bounded by a polynomial in $d$.
    Then the number of rational points of height at most $H$ on the graph of $f$ is bounded by a polynomial in $\log H$. 
\end{corollary}

\begin{proof}
By \cite[Proposition 2.4]{Pila:2006ab}, the set $\Ga(f)(\Q;H)$ of rational points of height at most $H$ in $\Ga(f)$ 
is contained in the union of at most $6B (\frac{(d+1)(d+2)}{2})^C\, H^{8/(d+3)}$ algebraic curves of degree $d$, where $d \ge 5$ and $B,C$ 
are constants depending on the class of $f$. (The Denjoy--Carleman class $\cC_{M_s}$ is contained in the class $M(B,C)$ 
of mild functions used in \cite{Pila:2006ab}.)

Let $d = \lfloor \log H \rfloor$, assuming that $H \ge e$.
Then $\Ga(f)(\Q;H)$ is contained in the union of at most
\[
    6B\, 3^C\, e^8 (\log H)^{2C}
\]
algebraic curves of degree $d$, using that $\frac{(d+1)(d+2)}{2} \le 3 (\log H)^2$ and $H^{8/(d+3)} \le e^8$.
Thus, the number of points in 
$\Ga(f)(\Q;H)$ is bounded by
\[
    6B\, 3^C\, e^8 (\log H)^{2C} \cdot \zz_{\lfloor \log H \rfloor}(f)
\]
which is bounded by a polynomial in $\log H$, by \Cref{c:Denjoy1}, provided that 
$\ee_d(f)$ is bounded by a polynomial in $d$. Thus \Cref{c:rat} is proved.
\end{proof}

When $\ee_d(f)$ is bounded by a polynomial in $d$?
By \Cref{r:sp}, there are real analytic $f$ for which this is not the case.
On the other hand, it holds for real analytic functions of the form $f = \sum_{j=1}^k p_j e^{q_j}$, where 
$p_j,q_j$ are real polynomials; see \cite[Section 2.7]{BRUDNYI:2015vu}. 
It seems to be difficult to identify \emph{concrete} non-analytic quasianalytic functions with this property. 

We will provide a family of functions $f_\xi : [0,1] \to [0,1]$ in $\cC_{M_s}$, for $s \in (0,1]$, 
parameterized by a sequence of independent random variables, each uniformly distributed in $[-1,1]$,
such that 
$f_\xi$ is almost surely non-analytic and transcendental, and $\ee_d(f)$ is almost surely bounded by a polynomial in $d$.

\begin{definition}[The functions $f_\xi$] \label[d]{d:fa}
 Let $\mu^{|\infty} = (\mu_j)_{j \ge 1}$ be any admissible weight (not necessarily quasianalytic) and $M_0=1$ such that 
 the associated function
\[
    \om_M(t) := \sup_{k \in \N} \log \frac{t^k}{M_k}, \quad t>0,
\]
satisfies $\om_M(t) = o(t)$ and $\om_M(2t) = O(\om_M(t))$ as $t \to \infty$. 
Let $\xi = (\xi_n)_{n \ge 1}$ be a sequence of independent random variables, each uniformly distributed on $[-1,1]$.
    Define
    \begin{equation}
        f_\xi(x) := \tfrac{1}{2} + \sum_{n\ge 1} \xi_n c_n \cos(2\pi n x), \quad \text{ where } c_n := \frac{1}{4 n^2 e^{\om_M(2\pi n)}}.
    \end{equation}
\end{definition}

\begin{remark}
Note that $\om_M(t) = o(t)$ as $t \to \infty$ means that $\cC_M$ contains the real analytic functions as 
a proper subclass. This condition is satisfied if the associated sequence $a_j$ (see \eqref{eq:aj}) tends to $\infty$.
Moreover, $\om_M(2t) = O(\om_M(t))$ as $t \to \infty$ is a technical condition which holds provided that $\liminf_{j \to \infty} \frac{\mu_{\ell j}}{\mu_j} >1$ for some 
positive integer $\ell$; see \cite[Lemma 12]{BMM07}.
Clearly, these conditions are satisfied by all Denjoy weights.
Moreover, $\log t = o(\om_M(t))$ as $t \to \infty$ by definition.
\end{remark}

\begin{proposition} \label[p]{p:as}
    The family $\{f_\xi\}$ has the following properties:
    \begin{enumerate}
        \item $f_\xi : [0,1] \to [0,1]$ belongs to $\cC_M$.
        \item Almost surely, $f_\xi$ is not real analytic. More generally, almost surely $f_\xi$ does not belong to any class $\cC_{L}$ with $\om_M(t) = o(\om_L(t))$ as $t \to \infty$, 
            where $L$ is any weight sequence as in \Cref{d:fa}.
        \item  Almost surely, $\ee_d(f_\xi)$ is bounded by a polynomial in $d$ of degree $7$ (whose coefficients depend on $\xi$). 
    \end{enumerate}
\end{proposition}

\Cref{p:as} will be proved in \Cref{ssec:as}. 
In the proof, we will see that the condition $\om_M(t) = o(t)$ as $t\to \infty$ is only used for (2). For (1) and (3), the weaker condition $\om_M(t)=O(t)$ as
$t \to \infty$ suffices. Thus (1) and (3) also hold for the analytic class $M=(j!)$.

\begin{theorem} \label{t:rp}
    Let $M = M_s$ a Denjoy weight for $s \in (0,1]$ with $(M_s)_0 =1$. 
    Then, for the family $\{f_\xi\}$,
    the number of rational points of height at most $H$ on the graph of $f_\xi$ is almost surely bounded by a polynomial in $\log H$. 
\end{theorem}

\begin{proof}
    This follows from \Cref{c:rat} and \Cref{p:as}. 
    Note that $f_\xi$ is almost surely transcendent, because if there exists a nonzero polynomial $p \in \R[X,Y]$ such that $p(x,f_\xi(x)) = 0$, for all $x \in [0,1]$, 
    then $f_\xi$ is real analytic. 
    Indeed, the $1$-periodic function $f_\xi$ belongs to $\cC_{M_s}(\R)$ and, by quasianalyticity, $p(x,f_\xi(x))=0$, for all $x \in \R$.
    By \cite[Theorem 2]{Siciak70} or \cite[Proposition VI.3.11]{Malgrange66}, $f_\xi$ is real analytic.
\end{proof}

\subsection{Proof of \texorpdfstring{\Cref{p:as}}{Proposition 5.3}}\label{ssec:as}

(1) It is easy to check that $f_\xi$ is of class $\cC_M$  
because (see e.g.\ \cite[Proposition 3.2]{Komatsu73})
\[
    \frac{(2\pi n)^j}{e^{\om_M(2\pi n)}} \le \sup_{t>0} \frac{t^j}{e^{\om_M(t)}} = M_j. 
\]
Since $\sum_{n\ge 1}  c_n \le \sum_{n\ge 1} \frac{1}{4n^2} = \frac{\pi^2}{24}  < \frac{1}{2}$, 
$f_\xi$ maps $[0,1]$ into $[0,1]$.

(2) If the $1$-periodic function $f_\xi$ is of class $\cC_L$, then there exist constants $C, c>0$ such that the Fourier coefficients $\xi_n c_n$ of $f_\xi$ 
satisfy
\begin{equation} \label{eq:CL}
    |\xi_n| c_n  \le C e^{-\om_L(c n)}, \quad n \ge 1.
\end{equation}
Indeed,
\begin{equation} 
    |\hat f(n)| \le \inf_{j\in \N}\frac{\|f_\xi^{(j)}\|_{[0,1]}}{(2\pi n)^j} \le \inf_{j\in \N}\frac{C\rh^j L_j}{(2\pi n)^j} = C e^{-\om_L(cn)}
\end{equation}
with $c = 2\pi/\rh$.
Since $\om_M(t) = o(\om_L(t))$, $\om_L(2t) = O(\om_L(t))$, and $\log t = o(\om_L(t))$ as $t \to \infty$, there exists $n_0 \in \N$ such that 
\[
    c_n = \frac{1}{4 n^2 e^{\om_M(2\pi n)}} \ge e^{-\frac{1}{2} \om_L(c n)}, \quad n\ge n_0.
\]
By the Borel--Cantelli lemma, almost surely $|\xi_n| \ge 1/2$ for infinitely many $n$.
Hence,
\begin{equation} \label{eq:ra}
    |\xi_n| c_n \ge \tfrac{1}{2} e^{-\frac{1}{2} \om_L(c n)},
\end{equation}
almost surely, for infinitely many $n \ge n_0$, contradicting \eqref{eq:CL}. This proves (2);
the case $\om_L(t) =t$ shows that almost surely $f_\xi$ is not real analytic.

(3)
Set $N = \dim \cP_2^d = \frac{(d+1)(d+2)}{2}$.
Let $x_1,\ldots,x_N$ be distinct points in $[0,1]$ and $A$ the $N \times N$ matrix with entries $x_k^i f_\xi(x_k)^j$, 
for $1 \le k \le N$ and $0 \le i+j \le d$. 
For a polynomial $p(x,y) = \sum_{i + j \le d} c_{ij} x^i y^j \in \cP_2^d$ with coefficient vector $c = (c_{ij})$,
we have $\|p\|_{[0,1]^2} \le  \sum_{i+j\le d}|c_{ij}| = \|c\|_1$ and 
\begin{align*}
    \|p\|_{\Ga(f_\xi)} \ge \|Ac\|_\infty \ge N^{-1/2} \|Ac\|_2 \ge N^{-1/2} \si_{N}(A) \|c\|_2 \ge N^{-1} \si_N(A) \|c\|_1,
\end{align*}
where $\si_1(A) \ge \cdots \ge \si_N(A)$ are the singular values of $A$.
Using that $|\det A| = \si_1(A)  \cdots  \si_N(A) \le \si_N(A)  \|A\|_2^{N-1} \le \si_N(A) N^{N-1}$, we conclude that 
\[
    \frac{\|p\|_{[0,1]^2}}{\|p\|_{\Ga(f_\xi)}} \le \frac{N^N}{|\det A|}.
\]
This implies
\begin{equation} \label{eq:eed}
    \ee_d(f_\xi) \le N \log N + \log\Big(\frac{1}{|\det A|}\Big).
\end{equation}

Let $\xi_{\le N} :=(\xi_1,\ldots,\xi_N)$ and $\xi_{>N} := (\xi_{N+1},\xi_{N+2},\ldots)$.
We keep $\xi_{>N}$ fixed. Then $\xi_{\le N}$ is uniformly distributed in $[-1,1]^N$. 
We will estimate $|\det A|$ for the choice
\begin{equation} \label{eq:xk}
    x_k := \frac{2k-1}{4(N+1)}, \quad 1 \le k \le N. 
\end{equation}
Set $y_k := f_\xi(x_k)$ and $y= (y_1,\ldots, y_N)$.
This gives a decomposition $y^\top = a^\top + B \xi_{\le N}^\top$, where $B = C \on{diag}(c_1,\ldots,c_N)$, 
$C = (\cos(2\pi n x_k))_{1 \le k,n \le N}$, and $a = (a_1,\ldots,a_N)$ 
with $a_k = \frac{1}{2} + \sum_{n >N} \xi_n c_n \cos(2\pi n x_k)$.
It follows that $y$ is uniformly distributed in the parallelepiped $K = a + B [-1,1]^N$ contained in $[0,1]^N$. 

Let us compute the volume $|K|$. To this end,
let $\hat C = (\hat C_{k,n})$ be the $(N+1) \times (N+1)$ matrix, 
where $\hat C_{k,n} = (\frac{2}{N+1})^{1/2} \cos(2\pi n x_k)$ if  $1 \le k \le N+1$ and $1 \le n \le N$
and $\hat C_{k,0} = (\frac{1}{N+1})^{1/2}$ if $1 \le k \le N+1$. Then $\hat C$ (being a discrete cosine transform of type IV) is orthogonal 
and $(\frac{2}{N+1})^{1/2}\, C$ is its minor 
with $(N+1)$-st row and $0$-th column removed. Thus
\begin{equation}
    |\det C| = \Big(\frac{N+1}{2} \Big)^{N/2} |\hat C_{N+1,0}| = \Big(\frac{N+1}{2} \Big)^{N/2}\Big(\frac{1}{N+1}\Big)^{1/2}.
\end{equation}
It follows that 
\begin{equation} \label{eq:volK}
    |K| = 2^N \Big(\frac{N+1}{2} \Big)^{N/2}\Big(\frac{1}{N+1}\Big)^{1/2} c_1 \cdots c_N.
\end{equation}

For $x_k$ given by \eqref{eq:xk},
consider the homogeneous 
polynomial 
in the variables $y=(y_1,\ldots,y_N)$, 
\[
    D(y) = \det (x_k^i y_k^j)_{1\le k \le N, 0\le i+j\le d}
    = \sum_{\si} \on{sgn} \si \prod_{k=1}^N x_k^{i_{\si(k)}}  y_k^{j_{\si(k)}},
\]
where $\si$ ranges over all bijections from $\{1,\ldots,N\}$ to $\{(i,j) : 0 \le i+j \le d\}$. 
The degree of $D(y)$ is $m:= \sum_{0 \le i + j \le d}j =\sum_{j=0}^d j(d+1-j) = \frac{1}{6}d(d+1)(d+2)$.
Let $b_e$ for $e= (e_1,\ldots,e_N)$ denote the coefficient of the monomial $\prod_{k=1}^N y_k^{e_k}$ in $D(y)$. 
Then $b_e$ is up to sign a product of Vandermonde determinants 
\[
    \prod_{j=0}^d \det(x_k^i)_{k \in R_j, 0 \le i \le d-j},  
\] 
where $R_j = \{k : e_k = j\}$. 
Thus, for the choice of $e$ which corresponds to $R_0$ consisting of the first $d+1$ points,  $R_1$ of the next $d$ points, and so on,
we find
\[
    |b_e| \ge (2(N+1))^{-m}.
\]
The coefficients of a univariate real polynomial $p$ of degree $d$ are bounded by $(8e)^{d} \|p\|_{[0,1]}$; indeed,
by comparison with the Chebyshev polynomial (see \cite[Theorem 16.3.3]{RS02}) and the Remez inequality, each coefficient 
is bounded by $e^d \|p\|_{[-1,1]} \le (8e)^{d} \|p\|_{[0,1]}$.
Since each $y_k$ has degree at most $d$ in $D(y)$, we find by iteration that 
$|b_e| \le (8e)^{d N} \|D\|_{[0,1]^N}$ and hence 
\begin{equation} \label{eq:lb1}
    \|D\|_{[0,1]^N} \ge (2(N+1))^{-m}(8e)^{-dN}.
\end{equation}

By the Remez inequality (see \cite{Brudnyi:1973un}),
\begin{equation} \label{eq:lb2}
    \|D\|_{[0,1]^N} \le \Big(\frac{4N}{|K|}\Big)^m \|D\|_K,
\end{equation}
as well as, for $E := \{y \in K : |D(y)| \le (8N)^{-m} \|D\|_K\}$, 
\begin{equation} 
    \|D\|_K \le \Big(\frac{4N |K|}{|E|}\Big)^m \|D\|_E \le \Big(\frac{|K|}{2\, |E|}\Big)^m  \|D\|_K
\end{equation}
implying that $|E| \le \frac{1}{2} |K|$. Thus 
\begin{equation} \label{eq:L2}
    \int_{K} |D(y)|^2 \, dy \ge \int_{K\setminus E} |D(y)|^2 \, dy \ge\frac{1}{2} |K| \, (8N)^{-2m} \|D\|_K^2.
\end{equation}
Combining \eqref{eq:volK}, \eqref{eq:lb1}, \eqref{eq:lb2}, and \eqref{eq:L2}, we find
\begin{equation}
    \|D\|_{L^2(K;\frac{\cL^N}{|K|})} =  \Big(\int_{K} |D(y)|^2 \, \frac{dy}{|K|} \Big)^{1/2} \ge  \frac{1}{\sqrt 2} (64 N^2(N+1))^{-m} (8e)^{-dN} (c_1 \cdots c_N)^m.
\end{equation}
Since $m \sum_{k=1}^N \log(\frac{1}{c_k}) \le m N (\log(4 N^2) + \om_M(2\pi N)) = O(d^7)$, using that $\om_M(t) = O(t)$ as $t \to \infty$,
this implies that
\begin{equation} \label{eq:d7}
    \log\Big(\big(\|D\|_{L^2(K;\frac{\cL^N}{|K|})} \big)^{-1} \Big) = O(d^7).
\end{equation}
Note that \eqref{eq:d7} holds uniformly in $a$: the constants depend only on $M$ through the use of $\om_M(t) = O(t)$ as $t \to \infty$.

By \cite[Theorem 8]{Carbery:2001aa}, there is an absolute constant $\al>0$ such that, for all $t>0$, 
\begin{equation}
    \frac{1}{|K|} \big|\big\{y \in K : |D(y)| \le t\, \|D\|_{L^{1/m}(K,\frac{\cL^N}{|K|})} \big\}\big| \le \al\, t^{1/m}
\end{equation}
and, by \cite[Theorem 1]{Carbery:2001aa}, there is an absolute constant $\be>0$ such that 
\[\|D\|_{L^{2}(K,\frac{\cL^N}{|K|})} \le (\be m)^m \|D\|_{L^{1/m}(K,\frac{\cL^N}{|K|})}.\]
Thus
\begin{align*}
    \frac{1}{|K|} \big|\big\{y \in K : |D(y)| \le t\, \|D\|_{L^{2}(K,\frac{\cL^N}{|K|})}\big\}\big| \le \al \be m\, t^{1/m}.
\end{align*}
Taking $t = e^{-d^4}$, we have 
\begin{align} \label{eq:un}
    \frac{1}{|K|} \big|\big\{y \in K : |D(y)| \le e^{-d^4} \|D\|_{L^{2}(K,\frac{\cL^N}{|K|})}\big\}\big| \le \al \be m\, e^{-d^4/m}.
\end{align}
Note that the bound on the right-hand side is independent of $a$.

Let now $Q_d(\xi) = D(f_{\xi}(x_1),\ldots,f_{\xi}(x_N))$ denote the determinant of the $N \times N$ matrix with entries $x_k^i f_\xi(x_k)^j$, 
for $1 \le k \le N$ and $0 \le i+j \le d$, i.e., of $A$, where the $x_k$ are give by \eqref{eq:xk}. 
Consider the event 
\[
    E_d = \{\xi : |Q_d(\xi)| \le  e^{-d^4} \|Q_d\|_{L^{2}(K,\frac{\cL^N}{|K|})}\}
\]
and recall that $K$ depends on $a$. By Tonelli's theorem and \eqref{eq:un}, the probability $P(E_d)$ of $E_d$ 
is bounded by $\al \be m\, e^{-d^4/m}$ so that $\sum_d P(E_d) < \infty$.
By the Borel--Cantelli lemma, 
almost surely in $\xi$, the event $E_d$ is failing for all $d \ge d_0(\xi)$. 
For such $\xi$ and $d \ge d_0(\xi)$, \eqref{eq:d7} implies
\begin{equation}
  \log \Big( \frac{1}{|Q_d(\xi)|} \Big) \le d^4 + O(d^7) 
\end{equation}
In view of \eqref{eq:eed}, this completes the proof.

\subsection*{Acknowledgement}

This research was funded in whole or in part by the Austrian Science Fund (FWF) DOI 10.55776/PAT1381823.
For open access purposes, the author has applied a CC BY public copyright license to any author-accepted 
manuscript version arising from this submission.

The author acknowledges that \Cref{p:as} and its proof strategy were suggested by Claude Sonnet 5.5 (Anthropic).
The author independently verified the proof and incorporated it; he takes full responsibility for the content of the paper.


\def\cprime{$'$}
\providecommand{\bysame}{\leavevmode\hbox to3em{\hrulefill}\thinspace}
\providecommand{\MR}{\relax\ifhmode\unskip\space\fi MR }
\providecommand{\MRhref}[2]{%
  \href{http://www.ams.org/mathscinet-getitem?mr=#1}{#2}
}
\providecommand{\href}[2]{#2}

\end{document}